\documentclass[a4paper]{amsart}

\usepackage[all]{xy}

\usepackage{amsmath}
\usepackage{amssymb}
\usepackage{amsfonts}
\usepackage{amsthm}
\usepackage{amscd}

\usepackage{mathtools}

\usepackage{pifont}
\usepackage{latexsym}

\usepackage{graphicx} 

\usepackage{comment}

\newcommand{\vl}{\models}
\newcommand{\yy}{\rightarrow}

\newcommand{\power}{\mathcal{P}}

\newcommand{\nsystem}{\mathcal{V}}

\newcommand{\iand}{\bigwedge}
\newcommand{\ior}{\bigvee}

\newcommand{\ua}{\mbox{$\uparrow$}}

\newcommand{\thn}{\ \Rightarrow\ }
\newcommand{\eq}{\ \Leftrightarrow\ }

\newcommand{\gm}{\Gamma}

\newcommand{\logic}[1]{\mathsf{#1}}
\newcommand{\plogic}[1]{\mathsf{#1}_{\ast}}
\newcommand{\ilogic}[1]{\mathsf{#1}_{\omega_{1}}}
\newcommand{\formula}[1]{\mathrm{#1}}
\newcommand{\iformula}[1]{\mathrm{#1}_{\omega_{1}}}
\newcommand{\predd}{\formula{CD}}
\newcommand{\predb}{\formula{BF}}

\newcommand{\infb}{\iformula{BF}}
\newcommand{\domain}{\forall x(\phi\lor q)\supset(\forall x\phi)\lor q}

\newcommand{\barcan}{\forall x\Box\phi\supset\Box\forall x\phi}
\newcommand{\ibarcan}{\iand_{i\in\omega}\Box p_{i}\supset\Box\iand_{i\in\omega} p_{i}}

\newcommand{\qflt}[1]{\mathcal{QF}_{#1}}

\newcommand{\nec}{\Box}

\newcommand{\nmodel}[1]{\mathcal{#1}}
\newcommand{\nval}[1]{v_{#1}}
\newcommand{\nass}{\mathcal{A}}
\newcommand{\nint}{\mathcal{I}}
\newcommand{\ndom}{\mathcal{D}}

\newcommand{\lcong}[1]{\cong_{\logic{#1}}}
\newcommand{\lind}[1]{A_{\logic{#1}}}

\newcommand{\variable}{\mathsf{V}}
\newcommand{\propvar}{\mathsf{Prop}}
\newcommand{\sformula}{\mathsf{Sub}}

\numberwithin{equation}{section}

\theoremstyle{plain}
\newtheorem{theorem}{Theorem}[section]
\newtheorem{lemma}[theorem]{Lemma}

\newtheorem{corollary}[theorem]{Corollary}

\theoremstyle{definition}
\newtheorem{definition}[theorem]{Definition}

\date{}

\begin{document}
\title[An extension of J\'{o}nsson-Tarski representation]
{An extension of J\'{o}nsson-Tarski representation and 
model existence in predicate non-normal modal logics}
\author{Yoshihito Tanaka}{Kyushu Sangyo University}

\maketitle

\begin{abstract}
In this paper, we give an extension of the J\'{o}nsson-Tarski representation 
theorem
for both normal and non-normal modal algebras so that it preserves 
countably many infinitary meets and joins. 
To extend the J\'{o}nsson-Tarski representation to 
non-normal modal algebras we consider neighborhood frames
instead of 
Kripke frames just as Do\v{s}en's duality theorem for 
modal algebras, and to deal with infinite meets and joins, we make use of 
Q-filters instead of prime filters. 
Then, we show that every predicate modal logic, whether it is 
normal or non-normal, has a model defined on a neighborhood frame with 
constant domains, and give completeness theorem for some predicate modal logics. 
We also show the same results for infinitary modal logics. 
\end{abstract}

\section{Introduction}

In this paper, we give an extension of the J\'{o}nsson-Tarski representation 
theorem
for both normal and non-normal modal algebras so that it preserves 
countably many infinitary meets and joins. 
To extend the J\'{o}nsson-Tarski representation to 
non-normal modal algebras we consider neighborhood frames
instead of 
Kripke frames just as Do\v{s}en's duality theorem for 
modal algebras \cite{dsn89}, and to deal with infinite meets and joins, we make use of 
Q-filters instead of prime filters. 
Then, we show a model existence theorem for every 
predicate modal logic and completeness theorem for some predicate modal logics. 
We also show the same results for infinitary modal logics. 
Since we deal with non-normal modal algebras, we call  
a Boolean algebra with unary operator $\nec$ 
by {\em modal algebra}, as does Do\v{s}en in \cite{dsn89}, and call a
modal algebra which satisfies 
$\nec 1=1$
and 
$\nec(x\land y)=\nec x\land \nec y$
by {\em normal modal algebra}. 
We define {\em modal logics} and {\em normal modal logics} in the same manner. 
That is, the class of modal logics in this paper includes both normal and 
non-normal modal logics.

It is well known as the 
J\'{o}nsson-Tarski representation \cite{jns-trs51} 
that for any normal modal algebra $A$, 
there exists a monomorphism
from $A$ to the dual algebra of the Kripke frame defined by the set of 
all prime filters of $A$ and 
the binary relation obtained from the modal operator of $A$. 
The J\'{o}nsson-Tarski representation is a strong tool in the researches
of propositional normal modal logics \cite{chg-zkh97,blc-rjk-vnm01}, 
but the naive application of these researches 
to predicate modal logics does
not work, as the embedding does not 
preserve infinite meets and joins.

In the case of Boolean algebras,  
Rasiowa-Sikorski \cite{rsw-skr50,rsw-skr63}
introduced the notion of {\em Q-filters} 
to strengthen the notion of prime filters 
and 
showed the {\em Rasiowa-Sikorski lemma}, which states that
two distinct points in a Boolean algebra 
can be separated by a Q-filter. 
From the Rasiowa-Sikorski lemma, it follows that 
for any Boolean algebra $A$ and 
any countable subset $S$ of $\power(A)$, there exists 
an embedding $f$ from $A$
into the power set algebra of the set of all Q-filters of $A$,  
which satisfies that $f(\iand X)=\iand f[X]$ and 
$f(\ior X)=\ior f[X]$ for all $X\in S$. 
Applying this result to the Lindenbaum algebra, 
an algebraic proof of the completeness theorem for the classical 
predicate logic is obtained \cite{rsw-skr50,rsw-skr63}.  
In the case of  Heyting algebras, 
Rauszer-Sabarski \cite{rsz-sbl74} showed 
that for any Heyting algebra $A$ and any countable subset 
$S\subseteq \power(A)$ which 
satisfies certain conditions, 
there exists an embedding $f$ 
from $A$ into 
the dual Heyting algebra of the intuitionistic Kripke frame 
defined on the set of all Q-filters of $A$,  
which preserves all infinite meets and joins in $S$. 
One of the conditions for the existence of the embedding is 
\begin{equation}\label{rsinfd}
\iand_{x\in X}(x\lor y)=\iand X\lor y
\end{equation}
for any $X\in S$ and $y\in A$. 
In fact, Rauszer-Sabarski \cite{rsz-sbl75bsl} proved that 
\eqref{rsinfd} is necessary for the Rasiowa-Sikorski lemma for Heyting algebras.  
When we interpret universal quantifiers in a formula as 
an infinite meets in Heyting algebras, 
\eqref{rsinfd} corresponds to the axiom 
$$
\predd=\domain
$$
of constant domains. 
By using the Rasiowa-Sikorski lemma for Heyting algebras given in \cite{rsz-sbl75bsl}, 
Ono-Rauszer 
\cite{ono-rsz82} showed that if 
$\predd$ is provable in 
a superintuitionistic logic $\logic{L}$,  
there exists an intuitionistic Kripke model $\nmodel{M}$ with constant domains
which satisfies that 
$\phi\in\logic{L}$ iff $\nmodel{M}\vl\phi$ for any formula $\phi$.

Then, 
Tanaka-Ono \cite{tnk-ono98} showed an extension of 
the J\'{o}nsson-Tarski representation for normal modal algebras 
which preserves countably many infinitary meets and joins:
For any modal algebra $A$ and any countable subset $S\subseteq\power(A)$ 
which satisfies 
\begin{equation}\label{rsinfb}
\nec\iand X=\iand_{x\in X}\nec x
\end{equation}
for any $X\in S$
and some other conditions, 
there exists an embedding 
from $A$ to the dual algebra of the Kripke 
frame defined by the set of all Q-filters of $A$, 
which preserves all infinite meets and joins in $S$. 
The equation \eqref{rsinfb} corresponds to the Barcan formula 
$$
\predb=\barcan  
$$
and its infinitary variant 
$$
\infb=\ibarcan. 
$$
It is shown in \cite{tnk-ono98} that 
for any universal propositional normal modal logic $\logic{L}$, 
the predicate modal logic defined by 
the least predicate extension of $\logic{L}$ and $\predb$ 
is complete with respect to the class $C$ of Kripke frames with constant domains, 
where $C$ is the class of Kripke frames which validate $\logic{L}$. 
It is also shown in \cite{tnk-ono98} that the same 
result as predicate modal logics and $\predb$ 
holds for infinitary modal logics and $\infb$.

In this paper, we introduce a similar extension 
of the J\'{o}nsson-Tarski representation for both 
normal and non-normal modal algebras, without assuming 
\eqref{rsinfb}. 
It is proved by Do\v{s}en \cite{dsn89} that the category 
of modal algebras is dually equivalent to the category 
of neighborhood frames,  
and it is known that there exists a neighborhood frames which refutes 
\eqref{rsinfb} (see, \cite{arlcst-pct06,mnr16}), while 
every dual algebra of a Kripke frame validates it. 
Therefore, we consider dual algebras of neighborhood frames defined 
on the set of Q-filters. 
Then we show a model existence theorem for predicate modal logics:  
For any
predicate modal logic $\logic{L}$, 
whether it is normal or not, 
there exists a neighborhood model $\nmodel{M}$ with constant domains 
which satisfies that $\phi\in\logic{L}$ iff $\nmodel{M}\vl\phi$ for any formula 
$\phi$. 
We also show that the least predicate modal logics which are axiomatized by some of the axioms
$$
\nec\top,\  
\nec(p\land q)\supset\nec p\land \nec q,\ 
\nec p\land\nec q\supset\nec(p\land q)
$$
are complete with respect to the corresponding classes of neighborhood frames with 
constant domains. 
The completeness of these logics are already given by
Arl\'{o}-Costa and Pauit in \cite{arlcst-pct06}, but we give another 
algebraic proof. 
As is shown in \cite{arlcst-pct06}, the Barcan formula is not valid in the 
class of neighborhood frames with constant domains. 
Indeed, the least predicate extension of $\logic{K}$, 
in which $\predb$ is not provable, is complete with 
respect to a class of neighborhood frames with constant domains. 
See \cite{arlcst02} for more about the relation between the Barcan formula and neighborhood frames. 
We also apply the representation theorem to infinitary modal logics, and 
show the model existence theorem for all infinitary modal logics. 
As a corollary, we give completeness theorem for some infinitary modal logics, 
including the least infinitary extension $\ilogic{K}$ of $\logic{K}$. 
The completeness theorem of $\ilogic{K}$ is already given by
Minari in \cite{mnr16}, 
but we extend it and give the same algebraic proof as predicate modal logics.

The outline of the paper is as follows: 
In Section~\ref{sec:pre}, we recall basic definitions. 
In Section~\ref{sec:rs}, we present an extension of 
the J\'{o}nsson-Tarski representation. 
In Section~\ref{sec:mex}, we show model existence theorem 
and completeness theorem for monotonic predicate modal logics. 
In Section~\ref{sec:non-monotonic} and Section~\ref{sec:infinitary},
we discuss non-monotonic predicate modal logics and 
infinitary modal logics, respectively.

\section{Preliminaries}\label{sec:pre}

In this section, we recall basic definitions.

Let $\langle W,\leq\rangle$ be a partially ordered set.  
For any $X\subseteq W$, 
we write 
$\ua X$ for the upward closure of $X$. That is, 
\[
\ua X=\{w\in W\mid \exists x\in X(x\leq w)\}. 
\]
Let $f\colon A\yy B$ be a mapping from a set $A$ to a set $B$. 
For any set $X\subseteq A$ and $Y\subseteq B$, 
$f\left[X\right]$ and $f^{-1}\left[Y\right]$ denote the sets
\[
f\left[X\right]
=
\{f(x)\mid x\in X\},\ \ 
f^{-1}\left[Y\right]
=
\{x\in X\mid f(x)\in Y\},  
\]
respectively.

\begin{definition}
A {\em neighborhood frame} is a pair
$\langle C, \nsystem\rangle$, where 
$C$ is a non-empty set and 
$\nsystem$ is a mapping from $C$ to $\power(\power(C))$. 
A neighborhood frame 
$\langle C, \nsystem\rangle$ is said to be 
{\em monotonic}, 
{\em topped} and  
{\em closed under finite intersections} 
(cufi, for short), if it satisfies the following conditions 
respectively: 
\begin{description}
\item[Monotonic]
for any $c\in C$, $\ua\nsystem(c)=\nsystem(c)$;
\item[Topped]
for any $c\in C$, $\nsystem(c)$ includes $C$; 
\item[Cufi]
for any $c\in C$ and any non-empty finite set $S$, if 
$S\subseteq\nsystem(c)$ then $\bigcap S\in \nsystem(c)$.  
\end{description}

Let $Z_{1}=\langle C_{1},\nsystem_{1}\rangle$ and $Z_{2}=\langle C_{2},\nsystem_{2}\rangle$
be neighborhood frames. 
A mapping $f\colon C_{1}\yy C_{2}$ is called a 
{\em homomorphism of neighborhood frames} from $Z_{1}$ to $Z_{2}$, if for 
any $c\in C_{1}$ and $X\subseteq C_{2}$, 
$$
f^{-1}[X]\in \nsystem_{1}(c)\eq X\in \nsystem_{2}(f(c))
$$
holds.  
\end{definition}

\begin{definition}
An algebra $\langle A;\lor,\land,-,\nec,0,1\rangle$ is called 
a {\em modal  algebra}, 
if its reduct $\langle A;\lor,\land,-,0,1\rangle$ is a Boolean algebra 
and $\nec$ is a unary operator on $A$. 
A modal algebra $A$ is said to be 
{\em monotonic}, 
{\em topped} and {\em cufi}, 
if it satisfies the following conditions, respectively:  
\begin{description}
\item[Monotonic]
for any $x$ and $y$ in $A$, 
$\nec (x\land y)\leq \nec x\land\nec y$; 
\item[Topped]
$\nec 1=1$; 
\item[Cufi]
for any $x$ and $y$ in $A$, 
$
\nec x\land \nec y\leq\nec(x\land y)
$. 
\end{description}
Note that 
each of monotonicity and cufi
can be defined by a single equation, and 
monotonicity is equivalent to 
\begin{equation*}
x\leq y\thn \nec x\leq\nec y.  
\end{equation*}
Let $A$ and $B$ be modal algebras. 
A mapping $f\colon A\yy B$ is called 
a {\em homomorphism of modal algebras}, if $f$ is a homomorphism 
of Boolean algebras which satisfies
$$
f(\nec x)=\nec f(x)
$$
for any $x\in A$. 
\end{definition}

\section{Rasiowa-Sikorski lemma and 
an extension of J\'{o}nsson-Tarski representation}\label{sec:rs}

Let $A$ be a Boolean algebra. A non-empty subset $F\subseteq A$ is 
called a {\em filter} of $A$, if it satisfies the following conditions: 
\begin{enumerate}
\item
$\ua F=F$;
\item
$x,\ y\in A\thn x\land y\in A$, for any $x$ and $y$ in $A$. 
\end{enumerate}
A filter is said to be {\em proper}, if $0\not\in F$. 
A proper filter $F$ is said to be {\em prime} if 
$x\lor y\in F$ then either $x\in F$ or $y\in F$, 
for any $x$ and $y$ in $A$.

\begin{definition}
{(Rasiowa-Sikorski \cite{rsw-skr50,rsw-skr63})}. 
Let $A$ be a Boolean algebra and   
$S\subseteq \power(A)$. 
A prime filter $F$ is said to be a {\em Q-filter for $S$}, if 
it satisfies that 
for any $X\in S$,  
$\iand X\in F$ 
whenever 
$\iand X\in A$ and 
$X\subseteq F$. 
\end{definition}

We write $\qflt{S}(A)$ for the set of all 
Q-filters for $S$. 
The following lemma is called the Rasiowa-Sikorski lemma 
\cite{rsw-skr50,rsw-skr63}.

\begin{lemma} \label{ba-ipft}
{\rm
(Rasiowa-Sikorski \cite{rsw-skr50,rsw-skr63})}. 
Let $A$ be a
Boolean algebra and  
$S$ a countable subset of $\power(A)$. 
Then
for any $a$ and $b$ in $A$ with $a\not\leq b$,    
there exists a Q-filter $F$ for $S$ 
such that $a\in F$ and $b\not\in F$. 
\end{lemma}

The following two lemmas define a neighborhood frame from 
a given modal algebra, and vice versa.

\begin{lemma}\label{matonfr}
Let $A$ be a monotonic modal algebra and 
$S$ a countable subset of $\power(A)$. 
Define a pair 
$J_{S}(A)$ by 
$$
J_{S}(A)=\langle \qflt{S}(A),\nsystem_{A}\rangle, 
$$
where, 
$$
\nsystem_{A}(F)=
\ua
\left\{
\left\{G\in\qflt{S}(A)\mid x\in G\right\}
\mid
\nec x\in F
\right\}
$$
for any $F$.  
Then, $J_{S}(A)$ is an monotonic neighborhood frame. 
Moreover, 
$J_{S}(A)$ is topped if $A$ is topped, and $J_{S}(A)$ is cufi
if $A$ is cufi.  
\end{lemma}

\begin{proof}
It is clear from the definition that $J_{S}(A)$ is an monotonic neighborhood frame. 
Suppose $\nec 1=1$ in $A$. 
Then, for any $F\in\qflt{S}(A)$, $\nec 1\in F$. Hence, 
$$
\qflt{S}(A)=\{G\mid 1\in G\}\in\nsystem_{A}(F). 
$$
Suppose $A$ is cufi. Take any 
$X$ and $Y$ in $\nsystem_{A}(F)$. Then, there exists $\nec x$ and $\nec y$
in $F$ such that 
$$
\{G\mid x\in G\}\subseteq X,\ 
\{G\mid y\in G\}\subseteq Y,
$$
respectively. 
For any $H\in\qflt{S}(A)$, 
\begin{align*}
H\in \{G\mid x\in G\}\cap\{G\mid y\in G\}
& \eq
x\in H,\ y\in H\\
& \eq
x\land y\in H\\ 
& \eq
H\in \{G\mid x\land y\in G\}. 
\end{align*}
Hence, 
$$
\{G\mid x\land y\in G\}\subseteq X\cap Y. 
$$
Therefore, $X\cap Y\in \nsystem_{A}(F)$,  
since $\nec x\land\nec y\leq\nec (x\land y)\in F$.  
\end{proof}

\begin{lemma}\label{nfrtoma}
Let $Z=\langle C,\nsystem\rangle$ be a
neighborhood frame. 
Define $K(Z)$ by 
$$
K(Z)=
\langle
\power(C);\cup,\cap,C\setminus-,\nec_{Z},\emptyset,C
\rangle, 
$$
where
$$
\nec_{Z}X=\{c\in C\mid X\in\nsystem(c)\}
$$
for any $X\subseteq C$.
Then, $K(Z)$ is a modal algebra and 
if $Z$ is monotonic, topped and cufi, 
then
$K(Z)$ is monotonic, topped and cufi, 
respectively.  
\end{lemma}

\begin{proof}
Suppose $Z$ is monotonic. 
Then, 
for any subsets $X$ and $Y$ of $C$,  
it is clear from the definition that $X\subseteq Y$ implies 
$\nec_{Z}X\subseteq \nec_{Z} Y$. 

Suppose $Z$ is topped. Then, $C\in\nsystem(c)$ for any 
$c\in C$. Hence, 
$\nec_{Z}C=C$. 

Suppose $Z$ is cufi. 
Take any subsets $S$ and $T$ of $C$. 
For any $c\in C$, 
\begin{align*}
c\in\nec_{Z}S\cap \nec_{Z}T
& \eq
S\in\nsystem(c),\ 
T\in\nsystem(c)\\
& \thn
S\cap T\in\nsystem(c)\\
& \eq
c\in\nec_{Z}(S\cap T)
\end{align*}
Hence, 
$
\nec_{Z}S\cap \nec_{Z}T
\subseteq
\nec_{Z}(S\cap T)$.  
\end{proof}

We call $K(Z)$ the {\em dual algebra} of a neighborhood frame $Z$. 
Now, we show an extension of the J\'{o}nsson-Tarski representation:

\begin{theorem}\label{extjt}
Let $A$ be a monotonic  modal algebra and  
$S$ a countable subset of  $\power(A)$. 
Then, a mapping 
$$
f\colon  A\yy K\circ J_{S}(A)
$$
defined by 
$$
f(x)=\{F\in\qflt{S}(A)\mid x\in F\}
$$
is a monomorphism of modal algebras which satisfies that 
for any $X\in S$, if $\iand X\in A$ then  
$f\left(\iand X\right)=\iand f[X]$. 
\end{theorem}

\begin{proof}
It is easy to see that $f$ is a homomorphism of Boolean algebras. 
We show $f$ preserves modal operator. 
For any $x\in A$, $f(\nec x)\subseteq \nec_{J_{S}(A)}f(x)$, since 
\begin{align*}
F\in f(\nec x)
& \eq 
\nec x\in F\\
& \thn
f(x)\in\nsystem_{A}(F)\\
& \eq 
F\in\nec_{J_{S}(A)} f(x). 
\end{align*}
We show converse. Suppose $\nec x\not\in F$. 
Take any $y\in A$ such that $\nec y\in F$. 
Then, $y\not\leq x$ by monotonicity of $A$.  
By Rasiowa-Sikorski lemma, there exists $G\in\qflt{S}(A)$ such that 
$y\in G$ and $x\not\in G$. 
Hence, $f(y)\not\subseteq f(x)$. Therefore, $f(x)\not\in\nsystem_{A}(F)$. 
Thus, $f$ is a homomorphism of  modal algebras.

Suppose $X\in S$ and $\iand X\in A$. 
For any $F\in\qflt{S}(A)$,  
\begin{align*}
F\in f\left(\iand X\right)
& \eq
\iand X\in F\\
& \eq
X\subseteq F &&
\text{($F$ is a Q-filter)}\\
& \eq
\forall x\in X(F\in f(x))\\
& \eq
F\in \bigcap_{x\in X}f(x). 
\end{align*}

\noindent
Take any $x$ and $y$ in $A$. If $x\not\leq y$, there exists a
Q-filter $F$ such that $x\in F$ and $y\not\in F$. 
Hence, $f(x)\not=f(y)$. Therefore, $f$ is a monomorphism of 
modal algebras.  
\end{proof}

\section{Model existence for predicate monotonic modal logics}\label{sec:mex}

The language we consider consists of the 
following symbols:
\begin{enumerate} 
\item
a countable set $\variable$ of variables;

\item
$\top$ and $\bot$;

\item logical connectives: 
$\land$,  $\neg$;

\item
quantifier:
$\forall$;

\item
for each $n\in\mathbb{N}$, 
countably many predicate symbols 
$P$, $Q$, $R$, $\cdots$ of arity $n$;

\item
modal operator
$\nec$.
\end{enumerate}

\noindent
The set $\Phi$ of formulas is 
the smallest set which satisfies:

\begin{enumerate}
\item
$\top$ and $\bot$ are in $\Phi$;

\item 
if $P$ is a predicate symbol of arity $n$ and 
$x_{1},\ldots,x_{n}$ are variables 
then
$P(x_{1},\ldots,x_{n})$ is in $\Phi$;

\item 
if $\phi$ and $\psi$ are in $\Phi$
then $(\phi\land\psi)$ 
is in $\Phi$;

\item 
if $\phi\in\Phi$ 
then $(\neg\phi)$ and 
$(\nec\phi)$  are in $\Phi$;

\item
if $\phi\in\Phi$ 
and $x\in\variable$ 
then $(\forall x\phi)\in\Phi$. 
\end{enumerate}

The symbols $\lor$, $\supset$,  $\exists$ are defined in a usual way.  
We write 
$\phi\equiv\psi$  and $\Diamond\phi$ for abbreviations of 
$(\phi\supset\psi)\land(\psi\supset\phi)$, and $\neg\nec\neg\phi$, 
respectively.

A {\em neighborhood model} for predicate modal logic is a four tuple 
$\langle C,\nsystem,\ndom,\nint\rangle$, 
where $\langle C,\nsystem\rangle$ is a neighborhood frame, 
$\ndom$ is a non-empty set called a {\em domain} and 
$\nint$ is a mapping 
called an {\em interpretation}
which maps 
each pair $(c,P)$, where $c\in C$ and 
$P$ is an $n$-ary predicate symbol, to an $n$-ary relation
$\nint(c,P)\subseteq \ndom^{n}$ over $\ndom$. 
An {\em assignment} $\nass$ for $\nmodel{M}$ 
is a mapping from the set $\variable$ of variables to 
$\ndom$. 
For each neighborhood model $\nmodel{M}=\langle C,\nsystem,\ndom,\nint\rangle$ and 
each assignment $\nass$, 
the valuation $\nval{\nass}$ of a formula $\phi$ on $\nmodel{M}$ is 
defined inductively, as follows:
\begin{enumerate}
\item
$\nval{\nass}(\top)=C$, 
$\nval{\nass}(\bot)=\emptyset$; 
\item 
for any $c\in C$,
any predicate $P$ of arity $n$ and any variables $x_{1},\ldots,x_{n}$, 
$c\in\nval{\nass}(P(x_{1},\ldots,x_{n}))$
iff
$(\nass(x_{1}),\ldots,\nass(x_{n}))\in \nint(c,P)$;
\item 
$\nval{\nass}(\phi\land\psi)=
\nval{\nass}(\phi)\cap\nval{\nass}(\psi)$;
\item 
$\nval{\nass}(\neg\phi)=
C\setminus \nval{\nass}(\phi)$;
\item 
$\nval{\nass}(\forall x\phi)=
\bigcap_{\nass'\in A_{x}}\nval{\nass'}(\phi)$, 
where $A_{x}$ is the set of all assignments for $\nmodel{M}$
which are different from $\nass$ only 
in the value of $x$;
\item 
$c\in\nval{\nass}(\nec\phi)$
$\eq$
$\nval{\nass}(\phi)\in\nsystem(c)$. 
\end{enumerate}

Let 
$\nmodel{M}=\langle C,\nsystem,\ndom,\nint\rangle$ be a neighborhood model 
and $\nass$ an assignment for $\nmodel{M}$. 
Take any formula $\phi$ and 
suppose that $\nass'$ is an assignment for $\nmodel{M}$ such that 
$
\nass(x)=\nass'(x)
$
for any $x\in\variable$ which occur freely in $\phi$. 
Then $\nval{\nass}(\phi)=\nval{\nass'}(\phi)$. 
Hence, if $\phi$ is a closed formula then $\nval{\nass}(\phi)=\nval{\nass'}(\phi)$ 
for any assignments $\nass$ and $\nass'$ for $\nmodel{M}$.

Let $Z=\langle C,\nsystem\rangle$ be a neighborhood frame. 
For any neighborhood model $\nmodel{M}=\langle C,\nsystem,\ndom,\nint\rangle$, 
any $c\in C$ and any formula $\phi$, 
we write 
$
\nmodel{M},c\vl\phi
$
if $c\in\nval{\nass}(\phi)$ for any assignment $\nass$ for $\nmodel{M}$. 
If 
$
\nmodel{M},c\vl\phi
$
for any $c\in C$, we write 
$
\nmodel{M}\vl\phi
$.
If 
$
\nmodel{M}\vl\phi
$ 
for any domain $\ndom$ and any interpretation $\nint$, 
we write 
$
Z\vl\phi
$. 
Let $C$ be a class of neighborhood frames. For any formula $\phi$, 
we write $C\vl\phi$ if $Z\vl\phi$ for any $Z\in C$.

A set $\logic{L}$ of formulas is called 
a {\em predicate modal logic}, if it 
contains classical predicate logic as a subset and closed under substitution, modus ponens 
and the following two inference rules: 
\begin{enumerate}
\item
for any $\phi\in\logic{L}$ and $x\in\variable$, 
if $\phi\in\logic{L}$ then $\forall x\phi\in\logic{L}$;
\item
for any $\phi$ and $\psi$ in $\logic{L}$, 
if $\phi\equiv\psi\in\logic{L}$ then $\nec\phi\equiv\nec\psi\in\logic{L}$. 
\end{enumerate}
A predicate modal logic $\logic{L}$ is said to be 
{\em monotonic}, {\em topped} and {\em cufi}, 
if 
$\nec(p\land q)\supset\nec p\land \nec q\in\logic{L}$,
$\nec\top\in\logic{L}$
and 
$\nec p\land\nec q\supset\nec(p\land q)\in\logic{L}$,  respectively. 
If a modal logic $\logic{L}$ is topped, then the necessitation rule
is admissible in $\logic{L}$. 
A predicate  modal logic $\logic{L}$ is said to be {\em normal} if it is
monotonic, topped and cufi. The least predicate normal modal logic
is the least predicate extension of $\logic{K}$, which we write 
$\plogic{K}$.

Let $\logic{L}$ be a predicate modal logic. Define 
a binary relation $\lcong{L}$ on the set $\Phi$ of all formulas by 
$\phi\lcong{L}\psi$ iff $\phi\equiv\psi\in\logic{L}$. 
Then, $\Phi/\lcong{L}$ is a well-defined  modal algebra, 
which is called the {\em Lindenbaum algebra} of $\logic{L}$. 
For any formula $\phi$, we write $|\phi|$ for the equivalence class of $\phi$. 
We claim that for any formula $\phi$ and any $x\in\variable$, 
$|\forall x\phi|=\iand_{y\in\variable}|\phi[x:=y]|$ holds in the Lindenbaum 
algebra of $\logic{L}$. Since $\forall x\phi\supset\phi[x:=y]$ is in $\logic{L}$ 
for any $y\in\variable$, 
the set of all equivalence classes of the shape $|\phi[x:=y]|$ 
has $|\forall x\phi|$ as its lower bound. 
We show $|\forall x\phi|$ is the greatest lower bound of this set. 
Suppose that $|\psi|$ is a lower bound of this set. Take 
a variable $z$ which does not occur in $\psi$ and $\phi$. 
Then, $\forall z(\psi\supset\phi[x:=z])\in\logic{L}$, 
since $\psi\supset\phi[x:=z]\in\logic{L}$. 
As $z$ does not occur in $\psi$, 
$\psi\supset\forall z(\phi[x:=z])\in\logic{L}$.   
Hence, $|\psi|\leq|\forall z(\phi[x:=z])|=|\forall x\phi|$. 
This complete the proof of the claim. 
The following is the model existence theorem for predicate monotonic  
modal logic.

\begin{theorem}\label{mdlex}
For any predicate monotonic  modal logic $\logic{L}$, there exists 
a monotonic neighborhood model $\nmodel{M}$ with constant domains which satisfies that 
\begin{equation}\label{modelex}
\phi\in\logic{L}\eq\nmodel{M}\vl\phi
\end{equation}
for any closed formula $\phi$. 
If $\logic{L}$ is topped and/or cufi, 
there exists a 
topped and/or cufi
monotonic neighborhood model $\nmodel{M}$ with constant domains
which satisfies \eqref{modelex}, respectively.
\end{theorem}

\begin{proof}
Let $\lind{L}$ be the Lindenbaum algebra of $\logic{L}$. 
Define $S\subseteq\lind{L}$ by 
$$
S=\{|\forall x\phi| \mid x\in\variable,\ \phi\in\Phi\}. 
$$
Define a neighborhood model 
$\nmodel{M}=\langle C,\nsystem,\ndom,\nint\rangle$
by $\langle C,\nsystem\rangle=J_{S}(\lind{L})$, 
that is, 
the neighborhood frame obtained from $\lind{L}$ in Lemma~\ref{matonfr}, 
$\ndom=\variable$ and 
$$
(x_{1},\ldots,x_{n})\in\nint(F,P)
\eq
|P(x_{1},\ldots,x_{n})|\in F
$$
for any variables $x_{1},\ldots,x_{n}$, any $n$-ary predicate $P$ and 
any Q-filter $F$ for $S$. 
Define an assignment $\nass$ for $\nmodel{M}$ by $\nass(x)=x$ for any $x\in\variable$. 
Then, for any formula $\phi$, it follows that 
$\nval{\nass}(\phi)=f(|\phi|)$, where $f$ is  a monomorphism given 
in Theorem~\ref{extjt}. 
Since $f$ is a monomorphism of modal algebras, 
\begin{align*}
\phi\in\logic{L}
& \eq
\text{$|\phi|=1$ in $\lind{L}$}\\
& \eq
\text{$f(|\phi|)=1$ in $K(J_{S}(\lind{L}))$}\\
& \eq
\nval{\nass}(\phi)=C. 
\end{align*}
If $\phi$ is a closed formula, the last equation is equivalent to 
$\nmodel{M}\vl\phi$. 
By Lemma~\ref{matonfr}, 
$J_{S}(\lind{L})$ is 
topped and/or cufi, whenever
$\logic{L}$ is topped and/or cufi, respectively. 
\end{proof}

\begin{corollary} (Arl\'{o}-Costa and Pauit \cite{arlcst-pct06}). 
The least (topped and/or cufi) predicate monotonic 
modal logic 
is sound and complete with respect to the class of 
(topped and/or cufi)
monotonic neighborhood frames with constant domains, respectively. 
\end{corollary}

\begin{proof}
Soundness follows from the leastness and completeness follows from 
Theorem~\ref{mdlex}. 
\end{proof}

Hence, $\plogic{K}$ is sound and complete with respect to the class 
$C$ of monotonic topped cufi neighborhood models with constant domains. 
It is well-known that the Barcan formula $\predb=\barcan$ is not a member of 
$\plogic{K}$. 
Indeed, it is shown in \cite{arlcst02} that 
$C\not\vl\predb$.

\section{Non-monotonic predicate modal logics}\label{sec:non-monotonic}

To apply the discussion in Section~\ref{sec:mex} to 
predicate non-monotonic  modal logic, 
we change the definition of $J_{S}(A)$ in Lemma~\ref{matonfr}
as follows:

\begin{lemma}
Let $A$ be a non-monotonic  modal algebra and 
$S$ a countable subset of $\power(A)$. 
Define a pair 
$\bar{J}_{S}(A)$ by 
$\bar{J}_{S}(A)=\langle \qflt{S}(A),\nsystem_{A}\rangle$,
where, 
$$
\nsystem_{A}(F)=
\left\{
\left\{G\in\qflt{S}(A)\mid x\in G\right\}
\mid
\nec x\in F
\right\}
$$
for any $F$.  
Then, $\bar{J}_{S}(A)$ is a neighborhood frame, and 
it is topped if $A$ is topped, and it is cufi
if $A$ is cufi
\end{lemma}

\begin{proof}
It is easy to see that is $\bar{J}_{S}(A)$ is a topped neighborhood frame 
whenever $A$ is topped. 
Suppose $A$ is cufi. Take any 
$X$ and $Y$ in $\nsystem_{A}(F)$. Then, there exists $\nec x$ and $\nec y$
in $F$ such that 
$$
X=\{G\mid x\in G\},\
Y=\{G\mid y\in G\},
$$
respectively. 
Since $\nec x\land\nec y\leq\nec (x\land y)\in F$,   
$$
X\cap Y=\{G\mid x\land y\in G\}\in\nsystem_{A}(F).
$$
\end{proof}

Then, we obtain the following representation theorem. 

\begin{theorem}\label{nextjt}
Let $A$ be a  non-monotonic  modal algebra and  
$S$ a countable subset of  $\power(A)$. 
Then, a mapping 
$$
f\colon  A\yy K\circ \bar{J}_{S}(A)
$$
defined by 
$$
f(x)=\{F\in\qflt{S}(A)\mid x\in F\}
$$
is a monomorphism of  modal algebras which satisfies that 
for any $X\in S$, if $\iand X\in A$ then  
$f\left(\iand X\right)=\iand f[X]$. 
\end{theorem}

\begin{proof}
We only show the case for modal operator. 
For any $x\in A$, $f(\nec x)\subseteq \nec f(x)$, since 
\begin{align*}
F\in f(\nec x)
& \eq 
\nec x\in F\\
& \thn
f(x)\in\nsystem_{A}(F)\\
& \eq 
F\in\nec_{\bar{J}_{S}(A)} f(x). 
\end{align*}
We show converse. Suppose $\nec x\not\in F$. 
Take any $y\in A$ such that $\nec y\in F$. 
Then, $y\not\leq x$ or $x\not\leq x$. 
By Rasiowa-Sikorski lemma, there exists $G\in\qflt{S}(A)$ such that 
$y\in G$ and $x\not\in G$ or there exists $H\in\qflt{S}(A)$ such that 
$y\not\in H$ and $x\in H$, respectively.  
Hence, $f(y)\not=f(x)$. Therefore, $f(x)\not\in\nsystem(F)$. 
\end{proof}

From Theorem~\ref{nextjt}, the model existence theorem 
and the completeness theorem for non-monotonic predicate modal logics 
follow by the 
same arguments as Section~\ref{sec:mex}.  
To summarize both monotonic and non-monotonic cases, we have the following: 

\begin{theorem}
For any 
(monotonic and/or topped and/or cufi) predicate
modal logic $\logic{L}$, there exists 
a 
(monotonic and/or topped and/or cufi, respectively) 
neighborhood model $\nmodel{M}$ with constant domains which satisfies that 
\begin{equation*}
\phi\in\logic{L}\eq\nmodel{M}\vl\phi
\end{equation*}
for any closed formula $\phi$. 
\end{theorem}

\begin{corollary} (Arl\'{o}-Costa and Pauit \cite{arlcst-pct06}). 
The least (monotonic and/or topped and/or cufi) predicate
modal logic 
is sound and complete with respect to the class of 
(monotonic and/or topped and/or cufi, respectively)
monotonic 
neighborhood frames with constant domains, respectively. 
\end{corollary}

\section{Infinitary modal logics}\label{sec:infinitary}

In this section, we apply the arguments in the previous two sections 
to infinitary modal logics.

The language we consider consists of the 
following symbols:
\begin{enumerate} 
\item
a countable set $\propvar$ of propositional variables;

\item
$\top$ and $\bot$;

\item logical connectives: 
$\iand$,  $\neg$;

\item
modal operator
$\nec$.
\end{enumerate}

The set $\Phi$ of formulas is 
the smallest set which satisfies: 
\begin{enumerate}
\item
$\propvar\subseteq\Phi$;

\item
$\top\in\Phi$ and $\bot\in\Phi$;

\item 
if $\gm$ is a countable subset of $\Phi$, 
then $\iand\gm\in\Phi$; 

\item 
if $\phi\in\Phi$ 
then $(\neg\phi)$ and 
$(\nec\phi)$  are in $\Phi$;
\end{enumerate}

The symbols $\supset$,  $\equiv$ and $\Diamond\phi$ 
are defined in the same way as Section~\ref{sec:mex}. 
For each countable set $\gm$ of $\Phi$
we write 
$\ior\gm$  for 
$\neg\iand_{\phi\in\gm}\neg\phi$. 

Let $\phi$ be any formula. 
The set $\sformula(\phi)$ of all subformulas of $\phi$ is defined inductively, 
as follows: 
\begin{enumerate}
\item
$\sformula(p)=\{p\}$,  
for all $p\in\propvar$; 

\item
$\sformula(\top)=\{\top\}$ and 
$\sformula(\bot)=\{\bot\}$;

\item 
$\sformula(\iand\gm)=\{\iand\gm\}\cup\bigcup_{\phi\in\gm}\sformula(\phi)$;

\item 
$\sformula(\nec\phi)=\{\nec\phi\}\cup\sformula(\phi)$. 
\end{enumerate}

A {\em neighborhood model} for infinitary modal logic is a triple
$\langle C,\nsystem,v\rangle$, 
where $\langle C,\nsystem\rangle$ is a neighborhood frame and 
$v$ is a mapping from $\propvar$  to $C$, which is called a {\em valuation}. 
For each valuation $v$, 
the domain $\propvar$ is extended to $\Phi$ in the following way: 
\begin{enumerate}
\item
$v(\top)=C$, 
$v(\bot)=\emptyset$; 
\item 
$v(\iand\gm)=\bigcap_{\phi\in\gm}v(\phi)$;
\item 
$v(\neg\phi)=
C\setminus v(\phi)$;
\item 
$c\in v(\nec\phi)$
$\eq$
$v(\phi)\in\nsystem(c)$. 
\end{enumerate}

Let $Z=\langle C,\nsystem\rangle$ be a neighborhood frame. 
For any neighborhood model $\nmodel{M}=\langle C,\nsystem,v\rangle$
for infinitary modal logic, 
any $c\in C$ and any formula $\phi$, 
we write 
$
\nmodel{M},c\vl\phi
$
if $c\in v(\phi)$. 
Other expressions 
$\nmodel{M}\vl\phi$, 
$Z\vl\phi$ and   
$C\vl\phi$ are 
defined in the same way as Section~\ref{sec:mex}.

A set $\logic{L}$ of formulas is called 
an {\em infinitary  modal logic}, if it 
contains classical propositional logic as a subset 
and closed under substitution, modus ponens 
and the following three inference rules: 
\begin{enumerate}
\item
if $\iand\gm\in\logic{L}$ then $\iand\gm\supset\phi\in\logic{L}$ for any 
$\phi\in\gm$;
\item
if $\psi\supset\phi\in\logic{L}$ for any $\phi\in\gm$ 
then $\psi\supset\iand\gm\in\logic{L}$;
\item
for any $\phi$ and $\psi$ in $\logic{L}$, 
if $\phi\equiv\psi\in\logic{L}$ then $\nec\phi\equiv\nec\psi\in\logic{L}$. 
\end{enumerate}
An infinitary  modal logic $\logic{L}$ is said to be 
{\em monotonic}, {\em topped} and {\em cufi}, 
if 
$\nec(p\land q)\supset\nec p\land\nec q\in\logic{L}$,
$\nec\top\in\logic{L}$
and 
$\nec p\land \nec q\supset\nec(p\land q)\in\logic{L}$,  respectively. 
An infinitary  modal logic $\logic{L}$ is said to be {\em normal} if it is
monotonic, topped and cufi. 
We write $\ilogic{K}$ for the least normal infinitary 
modal logic, which is the least infinitary extension of $\logic{K}$.

Let $\logic{L}$ be an infinitary modal logic. 
As we have 
conjunction of any countable sets of formulas, 
there exist uncountably many elements of the 
shape $|\iand\gm|$
in the Lindenbaum algebra $\lind{L}$ of $\logic{L}$. 
On the other hand, 
the monomorphism given in Theorem~\ref{extjt} and \ref{nextjt} 
can preserve only countably many infinitary meets.  
However, we have a restricted version of the model existence theorem, 
as follows: 

\begin{theorem}\label{imdlex}
Let $U$ be a countable set of formulas. 
For any infinitary
(monotonic and/or topped and/or cufi) 
modal logic $\logic{L}$, there exists 
a 
(monotonic and/or topped and/or cufi, respectively) 
neighborhood model $\nmodel{M}$ which satisfies that 
\begin{equation*}
\phi\in\logic{L}\eq\nmodel{M}\vl\phi 
\end{equation*}
for any $\phi\in U$. 
\end{theorem}

\begin{proof}
We only show the case that $\logic{L}$ is monotonic. 
Let $A_{U}$ be the subalgebra of $\lind{L}$ generated by 
the set 
$$
\bigcup_{\phi\in U}
\left\{
|\psi|
\mid 
\psi\in\sformula(\phi)
\right\}. 
$$
Since $U$ is countable, so is  
$A_{U}$. 
Define a subset $S$ of $\power(A_{U})$ by 
$$
S=
\left\{
\left\{
|\phi|\mid \phi\in\gm
\right\}
\mid
\left|\iand\gm\right|\in A_{U}
\right\}.
$$
Then, $S$ is also countable. 
We claim that for any $\left|\iand\gm\right|\in A_{U}$, 
$$
\left|\iand\gm\right|=
\iand_{\phi\in\gm} \left|\phi\right|. 
$$
For any $\phi\in\gm$, $|\iand\gm|\leq|\phi|$ holds, 
since $\iand\gm\supset\phi\in\logic{L}$ by definition of infinitary modal logic. 
Hence, $|\iand\gm|$ is a lower bound of the 
set $\left\{|\phi|\mid \phi\in\gm\right\}$. 
Suppose that $|\psi|\in A_{U}$ is a lower bound of this set. 
Then, $\psi\supset\phi\in\logic{L}$, for any $\phi\in\gm$.  
By definition of infinitary  modal logic, $\psi\supset\iand\gm\in\logic{L}$. 
Therefore, $|\psi|\leq |\iand\gm|$. 
This complete the proof of the claim. 
Define a neighborhood model $\nmodel{M}=\langle C,\nsystem,v\rangle$
by $\langle C,\nsystem\rangle=J_{S}(A_{U})$ and
$$
F\in v(p)\eq |p|\in F. 
$$
By the same argument as in the proof of Theorem~\ref{mdlex},  
we obtain that 
\begin{equation*}
\phi\in\logic{L}\eq\nmodel{M}\vl\phi 
\end{equation*}
for any $|\phi|\in A_{U}$, and 
that $\nmodel{M}$ is monotonic. 
\end{proof}

As a corollary, we obtain an extension of the 
completeness theorem for $\ilogic{K}$ given by Minari \cite{mnr16}, 
as follows:

\begin{corollary} 
The least (monotonic and/or topped and/or cufi) infinitary 
modal logic 
is sound and complete with respect to the class of 
(monotonic and/or topped and/or cufi, respectively)
monotonic 
neighborhood frames, respectively. 
\end{corollary}

\begin{proof}
We only show that the logic $\ilogic{K}$ is sound and complete with 
respect to the class $C$ of monotonic topped cufi neighborhood frames. 
Suppose $\phi\in\ilogic{K}$. 
Then, $C\vl\phi$, 
by the leastness of $\ilogic{K}$. 
Suppose $\phi\not\in\ilogic{K}$. 
Let $U=\{\phi\}$. 
By Theorem~\ref{imdlex}, 
there exists a 
monotonic topped cufi
neighborhood model $\nmodel{M}$ such that $\nmodel{M}\not\vl\phi$. 
Hence, $C\not\vl\phi$. 
\end{proof}

It is known that the formula 
$$
\infb=\ibarcan, 
$$
that is, an infinitary translation of the Barcan formula $\predb$, 
is valid in every Kripke model.  However, $\infb$ is not a member of $\ilogic{K}$. 
Hence, $\ilogic{K}$ is Kripke incomplete \cite{tnkpred, mnr16}.  
In contrast, $\infb$ is refuted in a neighborhood frame
\cite{mnr16}. 
The counter example in \cite{mnr16} is given by a multi-relational Kripke 
frame, but it is easy to obtain an equivalent neighborhood frame from it.

\bibliographystyle{plain}
\bibliography{myref.sjis}

\newcommand{\noop}[1]{}
\begin{thebibliography}{10}

\bibitem{arlcst02}
Horacio Arl\'{o}-Costa.
\newblock First-order extensions of classical systems of modal logic; the role
  of the {Barcan} schemas.
\newblock {\em Studia Logica}, 71:87--118, 2002.

\bibitem{arlcst-pct06}
Horacio Arl\'{o}-Costa and Eric Pacuit.
\newblock First-order classical modal logic.
\newblock {\em Studia Logica}, 84(2):171--210, 2006.

\bibitem{blc-rjk-vnm01}
Patrick Blackburn, Maarten de~Rijke, and Yde Venema.
\newblock {\em Modal Logic}.
\newblock Cambridge, third edition, 2001.

\bibitem{chg-zkh97}
Alexander Chagrov and Michael Zakharyaschev.
\newblock {\em Modal Logic}.
\newblock Oxford University Press, 1997.

\bibitem{dsn89}
Kosta Do\v{s}en.
\newblock Duality between modal algebras and neighbourhood frames.
\newblock {\em Studia Logica}, 48:219--234, 1989.

\bibitem{jns-trs51}
Bjarni J{\'o}nsson and Alfred Tarski.
\newblock {Boolean} algebras with operators {I}.
\newblock {\em American Journal of Mathematics}, 73:891--931, 1951.

\bibitem{mnr16}
Pierluigi Minari.
\newblock Some remarks on the proof-theory and the semantics of infinitary
  logic.
\newblock In Reinhard Kahle, Thomas Strahm, and Thomas Studer, editors, {\em
  Advances in Proof Theory}, pages 291--318. Springer, 2016.

\bibitem{ono-rsz82}
Hiroakira Ono and Cecylia Rauszer.
\newblock On an algebraic and {Kripke} semantics for intermediate logics.
\newblock In {\em Universal Algebra and Applications}, pages 431--438.
  PWN-Polish Scientific Publishers, 1982.

\bibitem{rsw-skr50}
Helena Rasiowa and Roman Sikorski.
\newblock A proof of the completeness theorem of {G{\"o}del}.
\newblock {\em Fundamenta Mathematicae}, 37:193--200, 1950.

\bibitem{rsw-skr63}
Helena Rasiowa and Roman Sikorski.
\newblock {\em The Mathematics of Metamathematics}.
\newblock PWN-Polish Scientific Publishers, 1963.

\bibitem{rsz-sbl74}
Cecylia Rauszer and Bogdan Sabalski.
\newblock Representation theorem for distributive pseudo-{Boolean} algebra.
\newblock {\em Bulletin of the Section of Logic}, 3:3/4:17--21, 1974.

\bibitem{rsz-sbl75bsl}
Cecylia Rauszer and Bogdan Sabalski.
\newblock Notes on the rasiowa-sikorski lemma.
\newblock {\em Bulletin of the Section of Logic}, 4/3:109--113, 1975.

\bibitem{tnkpred}
Yoshihito Tanaka.
\newblock {Kripke} completeness of infinitary predicate multi-modal logics.
\newblock {\em Notre Dame Journal of Formal Logic}, 40:326--340, 1999.

\bibitem{tnk-ono98}
Yoshihito Tanaka and Hiroakira Ono.
\newblock The {Rasiowa-Sikorski} lemma and {Kripke} completeness of predicate
  and infinitary modal logics.
\newblock In Michael Zakharyaschev, Krister Segerberg, Maarten de~Rijke, and
  Heinrich Wansing, editors, {\em Advances in Modal Logic}, volume~2, pages
  419--437. CSLI Publication, 2000.

\end{thebibliography}
\end{document}